\documentclass[12pt]{amsart}
\usepackage{amsmath,
amsfonts,
amsthm,
amssymb,
mathrsfs}
\usepackage{amsmath}
\usepackage{amsfonts}
\usepackage{amssymb}

\newcommand{\dom}{\operatorname{dom}}
\newcommand{\lex}{\operatorname{lex}}
\newcommand{\comment}[1]{}

\newcommand{\Ht}{\mathrm{ht}}
\newcommand{\CH}{\mathrm{CH}}
\newcommand{\GCH}{\mathrm{GCH}}

\newcommand{\Seq}[1]{\langle #1 \rangle}

\newcommand{\ZFC}{\mathrm{ZFC}}

\newcommand{\range}{\mathrm{range}}

\newcommand{\rest}{\upharpoonright}

\newcommand{\Bcal}{\mathcal{B}}
\newcommand{\Fcal}{\mathcal{F}}
\newcommand{\pphi}{\varphi}
\newcommand{\id}{\mathrm{id}}

\theoremstyle{plain}
\newtheorem{thm}{Theorem}[section]
\newtheorem{lem}[thm]{Lemma}

\theoremstyle{definition}
\newtheorem{defn}[thm]{Definition}

\begin{document}

\title{A minimal Kurepa line}
\author[H. Lamei Ramandi]{Hossein Lamei Ramandi}
\address{Institute f\"{u}r Mathematische Logik und 
Grundlagenforschung  \\
Westf\"{a}lische Wilhelms-Universit\"{a}t 
M\"{u}nster,  Germany}
\email{hlamaira@exchange.wwu.de}

\subjclass[]{}

\begin{abstract} 
We show it is consistent with $\ZFC$ 
that there is an everywhere Kurepa line which is order 
isomorphic to all of its dense $\aleph_2$-dense suborders.
Moreover, this Kurepa line does not contain any Aronszajn suborder.
We also show it is consistent with $\ZFC$ that there is a minimal Kurepa line which does not contain any Aronszajn suborder.
\end{abstract}

\maketitle

\section{Introduction}
Let $\kappa$ be a cardinal and $L$ be a linear order. 
$L$ is said to be \emph{$\kappa$-dense} if 
it has no end points and
every non-empty open interval 
has exactly $\kappa$-many elements. 
Baumgartner proved the following theorem in \cite{reals_isomorphic}.
\begin{thm}[\cite{reals_isomorphic}] \label{Baum}
It is consistent with $\ZFC$ that every two $\aleph_1$-dense sets of the reals are order isomorphic.
\end{thm}
If $L$ is a linear order, the \emph{density} of $L$ is the minimum 
cardinality of a dense $D \subset L$.
Baumgartner viewed his theorem as an extension of Cantor's 
famous theorem about $\aleph_0$-dense linear orders: 
$\aleph_0$-dense linear orders are order isomorphic.

The analogue of Baumgartner's theorem for linear orders of density $\aleph_1$ is closely related to Kurepa lines.
A linear order $L$ of cardinality $\aleph_2$ is said to be \emph{Kurepa}
if the closure of countable subsets of $L$ are countable and the density 
of $L$ is $\aleph_1$.
Note that Kurepa lines exist if an only if there are Kurepa trees. 
More precisely, if $T$ is a lexicographically ordered  Kurepa tree
then the set of all cofinal branches of $T$ with the lexicographic order of $T$ is a Kurepa line.

It is a well known classical fact that every two complete linear orders
of density $\aleph_0$ without endpoints are isomorphic.
In order to obtain Theorem \ref{Baum}, 
it  suffices to show that every two dense $\aleph_1$-dense subsets 
of the reals are order isomorphic.
This is the motivation behind the following theorem. 
\begin{thm}\label{main}
It is consistent with $\ZFC$ that there is a Kurepa line which is isomorphic to all of its dense $\aleph_2$-dense suborders. Moreover, this Kurepa line has no Aronszajn suborder. 
\end{thm}

The work involved in proving Theorem \ref{main} can be modified 
to show the following theorem.

\begin{thm}\label{min}
It is consistent with $\ZFC$ that there is a minimal Kurepa line which does not contain any Aronszajn suborders.
\end{thm}
It is worth mentioning that the minimal Kurepa type $K$ of Theorem
\ref{min} has the following property. Whenever $K = X \cup Y$ is a partition such
that every non-empty open subset of $K$ has exactly $\aleph_2$-many elements
of both pieces of the partition then $X$ and $Y$ are order isomorphic. Note that $K$ is $\omega_1$-metrizable and Lindel\"of of size $\aleph_2$ whose topology is generated by the order from $K$.
Such spaces were  considered as a potential 
higher analogue of the unit interval in \cite{Sikorski}.
The partitions we mentioned above represent a 
behavior of $K$ which stands in total contrast with the properties of the unit
interval of the reals.

Our work can be viewed as a continuation of \cite{second} and \cite{third}.
The work in \cite{second} shows it is consistent with $\ZFC$ that 
there is a Kurepa tree which is minimal with respect to club embeddings 
and which contains no Aronszajn subtrees. 
It is not clear of the forcing notions in \cite{second} will preserve $\omega_1$ if 
we require that the trees are lexicographically ordered and the embeddings 
preserve this order.
This problem was resolved in \cite{third}. 
However, in order to preserve $\omega_2$ some extra structure 
was added to the tree. 
The added structure only allows to add embeddings between 
dense subsets of size $\aleph_1$.
Here we preserve $\aleph_1, \aleph_2$ in a more flexible way.

\section{Preliminaries}
In this section we review some facts and terminology regarding $\omega_1$-trees, liner orders and countable support iteration of some type of forcing notions.
The material in this section can be found in \cite{second} and \cite{no_real_Aronszajn}.

For more clarity we fix some notation and terminology.
A tree $T$ is said to be an $\omega_1$-tree if it has countable levels and $\Ht(T)=\omega_1$. 
A subtree of a tree is a subset which is downward closed.
Assume $T$ is an $\omega_1$-tree.
If $t \in T$, the $T_t$ is the set of all elements in $T$ that are comparable with $t$. 
We refer to the forst level of $T_t$ which is not a singleton by $t^+$.
$T_t$ is also called a cone of $T$.
We say $T$ is pruned if it has no countable cone. 
If $\alpha$ is an ordinal, $T(\alpha)$ is the set of all elements in $T$ whose 
height is $\alpha$.
A branch in $T$ is a downward closed chain in $T$.
A branch is said to be cofinal if it intersects all levels of $T$.
$\Bcal(T)$ is the collection of all cofinal 
branches of $T$.
If $b$ is a branch then $b(\alpha)$ is  the unique element in $b \cap T(\alpha)$.
If $t \in T$ and $\alpha \leq \Ht(t)$ then $t(\alpha) = b(\alpha)$ for some (any)
branch $b$ with $t \in b$.
If $A$ is a set of ordinals then $T \rest A$ is the collection of all elements of 
$T$ whose height is in $A$. The order on $T \rest A$ is the same as $T$.
A tree $T$ is said to be \emph{normal} if whenever $s,t$ are two distinct elements 
of the same limit height then they have different sets of predecessors.
Assume $a,b$ are among elements and cofinal branches of $T$.
$\Delta(a,b)$ is the smallest $\alpha$ such that $b(\alpha) \neq a(\alpha)$.
In particular $\Delta(a,b)$ is never a limit ordinal in normal trees.

Assume $T$ is a lexicographically ordered $\omega_1$-tree, 
$U$ is a downward closed subtree of $T$ and $b \in \Bcal(U)$. 
We say 
$b$ is a local right end point of $U$ if there is $t \in b$ such that
$b(\alpha) = \max(U_t(\alpha), <_{\lex})$,
 for all 
$\alpha \in \omega_1 \setminus \Ht(t)$.
Observe that this definition implicitly introduces 
a one-to-one function from the set of all right end points of $U$ to 
$T$, which implies that there are at most $\aleph_1$ many $b \in \Bcal(U)$ that are local right end points of $U$.  
Note that if $b \in L \subset \Bcal(T)$ is  a local right end point of 
$\bigcup L$
then there is $t \in b$ such that $b = \max(L \cap \Bcal(T_t), <_{\lex})$. Analogous definitions and observations can be made for local left end points. If $b \in \Bcal(U)$ is a local right or left end point of $U$, it is called a local end point.

Fix a regular cardinal $\theta$. $H_\theta$ is the collection of all sets of hereditary cardinality less than $\theta$.
We always consider $H_\theta$ with a fixed well ordering without mentioning it.
Assume $\mathcal{P}$ is a forcing notion and $\theta$ is a regular cardinal such that $\mathcal{P}$ and the powerset
of $\mathcal{P}$ are in $H_\theta$. A countable elementary submodel $N$ of $H_\theta$ is said to be \emph{suitable} for $\mathcal{P}$
if $\mathcal{P} \in N$. A decreasing sequence $\langle p_n : n \in \omega \rangle$ of elements of $\mathcal{P}\cap N$ 
is said to be \emph{$(N,\mathcal{P})$-generic} if 
for all dense subsets $D$ of $\mathcal{P}$ that are in $N$ there is an $n \in \omega$ such that $p_n \in D$.
When $T$ is a normal $\omega_1$-tree, $\Omega(T)$ is the collection 
of all $M \cap \Bcal(T)$ such that $M$  is suitable for $T$ and 
for every $t \in T(M \cap \omega_1)$ there is $b \in M \cap \Bcal(T)$ with 
$t \in b$.


\begin{lem}\label{rational_copy}
Assume $T$ is a lexicographically ordered $\omega_1$-tree 
such that 
$(v^+, <_{\lex})$ is isomorphic to the rationals
for all $v \in T$.
Also assume $U$ is a downward closed pruned subtree of $T$, $t \in U$ and $U$ has no local end points.
Let $\theta > (2^{\omega_2})^+$ be a regular cardinal and 
$M \prec H_\theta$ be countable such that, $\{U,T,t \} \in M$ and $\delta = M \cap \omega_1.$
Then ($U_t(\delta), <_{\lex}$) is isomorphic to the rationals.
\end{lem}
\begin{proof}
Let $s_0, s_1$ be in $U_t(\delta)$. We show there is 
$s \in U_t(\delta)$ such that $s_0<_{\lex} s <_{\lex} s_1$.
Assume for a contradiction that there is no such $s$.
Let $u_0<_T s_0$ and $u_1 <_T s_1$ 
be distinct elements of the same height.
Since $U$ is pruned, for every $\alpha \in \delta \setminus \Ht(u_0)$
the linear order
$(U_{u_0}(\alpha), <_{\lex})$ has a maximum.
Moreover, for $\alpha < \beta$ in $\delta \setminus \Ht(u_0)$
$$\max(U_{u_0}(\alpha), <_{\lex})=s_0(\alpha) <_T s_0(\beta) =\max(U_{u_0}(\beta), <_{\lex}).$$
So by elementarity, there is a cofinal branch $b_0 \in M \cap \Bcal(U)$ with $s_0 \in b_0$.
Similarly, there is a cofinal branch $b_1 \in \Bcal( U) \cap M$ such that $s_1 \in b_1$.

Let $\Delta(b_0, b_1) \in \delta_0 \in \delta $
and $u = b_1(\delta_0)$.
By elemetarity and the fact that $b_1$ is not a local end point, there is a $v \in M \cap U$  such that $u <_T v$ and  $v <_{\lex} b_1(\Ht(v))$.
Since $U$ is pruned there is $s$ above $v$ in $U(\delta)$.
But then 
$s_0=b_0(\delta) <_{\lex} s  <_{\lex} b_1(\delta) = s_1.$
Similar argument shows that $U_t(\delta)$ has no end points. 
\end{proof}

Let us review some facts and definitions from \cite{second}.

\begin{defn}[\cite{second}]
Assume $X$ is uncountable and $S \subset [X]^\omega$ is stationary. A poset $\mathcal{P}$ is 
said to be \emph{$S$-complete}
if every descending $(M, \mathcal{P})$-generic
sequence $\langle p_n: n\in \omega \rangle$ has a lower bound, for all $M$ with $M \cap X \in S$ and
$M$ suitable for $X,\mathcal{P}$.
\end{defn}

\begin{lem}[\cite{second}] \label{iteration}
Assume $X$ is uncountable and $S\subset[X]^\omega$ is stationary. Then $S$-completeness is 
preserved under countable support iterations.
\end{lem}

\begin{lem}[\cite{second}]\label{No A subtree} 
Assume $T$ is an $\omega_1$-tree 
which has uncountably many cofinal branches and which has no Aronszajn subtree 
in the ground model $\mathbf{V}$. Also assume
 $\Omega(T)\subset [\mathcal{B}(T)]^\omega$ is stationary  and
$\mathcal{P}$ is an $\Omega(T)$-complete forcing.
Then  $T$ has no Aronszajn subtree in $\mathbf{V}^\mathcal{P}$.
\end{lem}

\begin{lem}[\cite{second}] \label{No New branch}
 Assume $T$ is an $\omega_1$-tree, $X$ is an uncountable set, 
$S \subset [X]^\omega$ is stationary, and 
$\mathcal{P}$ is an $S$-complete forcing. Then $\mathcal{P}$ does not add new cofinal branches to $T$.
 \end{lem}

\begin{defn}[\cite{second}] \label{S-cic}
Assume $S,X$ are as above and $\kappa$ is a regular cardinal. 
We say that $\mathcal{P}$ satisfies the \emph{$S$-closedness isomorphism condition for 
$\kappa$},
or $\mathcal{P}$ has the \emph{$S$-cic for $\kappa$},  if whenever
\begin{itemize}
\item
$M,N$ are suitable models for $\mathcal{P}$,
\item
both $M \cap X,$ $ N\cap X$ are in $S$,
\item
$h:M\rightarrow N$ is an isomorphism such that $h\upharpoonright (M\cap N) =\mathrm{ id}_{(M \cap N)}$,
\item
$\min((N \setminus M)\cap \kappa)> \sup(M\cap \kappa)$, and 
\item
$\langle p_n: n\in \omega \rangle $ is an $(M,\mathcal{P})$-generic sequence,
\end{itemize} 
then there is a common lower bound $q \in \mathcal{P}$ for
$\langle p_n: n\in \omega \rangle $ and $\langle h(p_n): n\in \omega \rangle $.
\end{defn}

\begin{defn}\label{Q}
Assume $|\Lambda| =\aleph_1$.
Let $Q$ be the forcing notion consisting of all
$p=(T_p, b_p)$ such that:
\begin{itemize}
\item $T_p$ is a lexicographically ordered countable tree of height 
$\alpha_p +1$ whose underlying set is a subset of $\Lambda$,
\item if $\nu \in \lim(\omega_1)$  and $s \neq t$ are in $T_p(\nu)$
then $s,t$ have different sets of predecessors, 
\item for all $t \in T_p$ there is $s \in T_p(\alpha_p)$ such that $t\leq_{T_p} s$,
\item if $t \in T_p$ and $\Ht(t)< \alpha_p$ then 
$(\{ s \in T_p(\Ht(t) +1): t <_{T_p}s \}, <_{\lex})$ is a countable dense linear order, and
\item $b_p$ is a countable partial function from $\omega_2$ to $T_p(\alpha_p)$.
\end{itemize}
We let $p \leq q$ if 
\begin{itemize}
\item $T_p \rest \alpha_q = T_q$,
\item $\dom(b_p) \supset \dom(b_q)$, and
\item $b_q(\xi) \leq b_p(\xi)$ for all $\xi \in \dom(b_q)$.
\end{itemize}
\end{defn}

It is well known that $Q$ is countably closed.
Moreover, if $\CH$ holds then $Q$ has the $\aleph_2$-cc.
Assume $\CH$ holds in $\textsc{V}$ and let $G$ be $\textsc{V}$-generic for $Q$.
Then $G$ introduces  an $\omega_1$-tree $T =\bigcup_{q \in G} T_q$ together 
with $\aleph_2$-many cofinal branches  $b_\xi:=\{b_q(\xi): q \in G \}$ for $\xi \in \omega_2$.
Note that $T$ has no Aronszajn subtree and $\Seq{b_\xi : \xi \in \omega_2}$
enumerates the set of all cofinal branches of  $T$ in $\textsc{V}[G]$.
Conversely, $G$ is uniquely characterized by $T, <_{\lex},$ and $\Seq{b_\xi: \xi \in \omega_2}$.
If $b \in \Bcal(T)$ let $\iota(b)$ be the $\xi \in \omega_2$ such that $b = b_\xi$.

\section{proof of theorems}
\begin{defn}
Assume $\pphi$ is a partial function from $\omega_2$ to $\omega_2$ and 
$C \subset \omega_2$ is a club.
We say 
$\pphi$ respects $C$ if for all $\alpha \in C$ and $\xi \in \dom(\pphi)$, 
$$\xi \in \alpha \longleftrightarrow \pphi(\xi) \in \alpha.$$
\end{defn}
\begin{defn}\label{fast}
Assume 
$X \in H_{(2^{\omega_2})^+}$. 
A club $C \subset \omega_2$ is said to be fast for $X$
if there are a regular cardinal $\lambda > (2^{\omega_2})^+$
and a continuous $\in$-chain 
$\Seq{M_\xi : \xi \in \theta}$
of  elementary submodels of $H_\lambda$
such that:
\begin{itemize}
\item$X \in M_0$, 
\item $|M_\xi|= \aleph_1$ for all $\xi \in \omega_2$,
\item $\xi \cup \omega_1 \subset M_\xi$ for all $\xi \in \omega_2$, 
\item $\Seq{M_\eta: \eta \in \xi} \in M_{\xi+1}$ for all $\xi \in \omega_2$, and
\item $C = \{ M_\xi \cap \omega_2 : \xi \in \omega_2\}$.
\end{itemize} 

\end{defn}
\begin{defn}\label{F}
Let $\kappa \in  \{\omega_1, \omega_2 \}$ and let
$T, <_{\lex},\Seq{b_\xi: \xi \in \omega_2}$ be the lexicographically ordered tree with the enumeration of its cofinal branches that is introduced by the generic filter of $Q$.
Let $X,Y$ be subsets of $\Bcal(T)$ such that: 
\begin{itemize}
\item[(a)]  for all $t \in U:= \bigcup X$, $|\Bcal(U_t) \cap X| = \kappa$ and $|\Bcal(U_t) \setminus X| = \omega_2$,
\item[(b)]  for all $t \in V:= \bigcup Y$, $|\Bcal(V_t) \cap Y| = \kappa$ and $|\Bcal(V_t) \setminus Y| = \omega_2$,
\item[(c)] $U$ and $V$ have no local end points.
\end{itemize}
Assume $C \subset \omega_2$ is a club which is fast for $T, X, Y$ and 
$\Seq{b_\xi : \xi \in \omega_2}$.
$\Fcal_{X,Y}(T)$ is the poset consisting of all 
$p=(f_p, \psi_p)$
such that:
\begin{enumerate}
\item The function $f_p: U\rest A_p \longrightarrow V\rest A_p$ is a 
$<_{\lex}$-preserving level preserving tree isomorphism, 
where $A_p \subset \omega_1$ is countable and closed with 
$\max(A_p) = \alpha_p$.  
\item \label{1-1} $\psi_p$ is a countable partial one-to-one function from $\iota[\Bcal(U)]$ to $\iota[ \Bcal(V)]$ such that $\pphi_p=\iota^{-1} \circ \psi_p \circ  \iota $ preserves $<_{\lex}$.
\item\label{respect}  The map $\psi_p$ respects the club $C$. 
\item\label{restriction} For all $b \in \dom(\pphi_p)$, $b \in X$ if and only if $\pphi(b) \in Y$.
\item\label{finite} For all $t \in T(\alpha_p)$ there are at most finitely many 
$b \in \dom(\pphi_p)$ with $t \in b$.
\item\label{con} For all $b \in \dom(\pphi_p)$, 
$f_p(b(\alpha_p))= [\pphi_p(b)](\alpha_p)$.
\end{enumerate}
We let $q \leq p$ if $A_q \cap \alpha_p = A_p$, $f_q \supset f_p$, and 
$\pphi_q \supset \pphi_p$. We write $\Fcal_{X,Y}$ or just  $\Fcal$ instead of $\Fcal_{X,Y}(T)$
if there is no ambiguity.
\end{defn}

It is obvious from Definition \ref{F} that if $q \leq p$ then $(f_q, \psi_p) \leq p$.
We sometimes use this fact without mentioning it.
It is also obvious that every condition $p$ is uniquely determined 
by $f_p$ and $\pphi_p$. 
In order to refer to a condition $p \in \Fcal$
we sometimes abuse the notations and use 
$\pphi_p$ -- which is a map between the set of branches --
instead of $\psi_p$ 
which is a map between ordinals.
\begin{lem}\label{height}
Let $T,X,Y,U,V,\Fcal$ be as in Definition \ref{F} and
$\alpha \in \omega_1.$ 
Let $D(\alpha)$ be the set of all conditions  $q \in \Fcal$ with $\alpha \leq \alpha_q$. 
Then $D(\alpha)$ is a dense subset of $\Fcal$.
\end{lem}
\begin{proof}
Let $\alpha \in \omega_1$ and $ p \in \Fcal$. 
Without loss of 
generality assume that for all $b,b' \in \dom(\pphi_p) \cup \range(\pphi_p)$, $\Delta (b,b') < \alpha$.
Condition (c) of Definition \ref{F} and Lemma \ref{rational_copy} imply that there is $\beta \in \omega_1 \setminus \alpha$ such that for all 
$t \in U(\alpha_p) \cup V(\alpha_p)$,
$(U_t(\beta), <_{\lex})$ and  $(V_t(\gamma), <_{\lex})$ are isomorphic to the rationals.

For every $t \in U(\alpha_p)$ we let $g_t : U_t(\beta) \longrightarrow V_{f_p(t)}(\beta)$ be a $<_{\lex}$-isomorphism such that 
if $b \in \dom(\pphi_p)$ and $t \in b$ then $g_t(b(\beta))=[\pphi_p(b)](\beta)$. 
By Condition \ref{con} of Definition \ref{F}, $[\pphi_p(b)](\beta) \in V_{f_p(t)}$. 
By Condition \ref{1-1}, 
the map $b(\beta) \mapsto [\pphi_p(b)](\beta)$
is order preserving.
By Condition \ref{finite}, the  sets $\{b(\beta): b \in \dom(\pphi_p) \wedge t \in b\}$
and $\{[\pphi_p(b)](\beta): b \in \dom(\pphi_p) \wedge t \in b\}$ are finite for every $t \in U_{\alpha_p}$.
Therefore, the function $g_t$ exists for every $t \in U(\alpha_p)$. 

Now define $f = \bigcup \{g_t: t \in U(\alpha_p)\}$, $A_q= A_p \cup \{ \beta \}$, 
$\pphi_q = \pphi_p$ and $f_q=f_p \cup f$.
Then  $f_q: U \rest A_q \longrightarrow V\rest A_q$ is a $<_{\lex}$-preserving level preserving tree isomorphism. 
It is easy to see that the rest of the conditions of Definition \ref{F} holds for $q=(f_q, \pphi_q)$ and it extends $p$.
\end{proof}

Lemma \ref{height} shows that $\Fcal_{X,Y}$ adds a club isomorphism from $\bigcup X$ to  $\bigcup Y$,
provided that $\omega_1$ is preserved.

\begin{lem}\label{branchDensity}
Let $T, X,Y,U,V,\Fcal$ be as in Definition \ref{F}.
Let $\eta \in \omega_2$ and $b_\eta \subset U$ $(b_\eta \subset V)$. 
Let $D_\eta $ $(D^\eta)$be the collection of all conditions $q \in \Fcal$ such that $b_\eta \in \dom(\pphi_q)$ $(b_\eta \in \range(\pphi_q))$.
Then $D_\eta$ $(D^\eta)$ is a dense subset of $\Fcal.$
\end{lem}
\begin{proof}
Due to similarity we only show $D_\eta$ is dense. 
Fix $p \in \Fcal.$
Let $C \subset \omega_2$ be the club which is respected by the second coordinates of conditions in $\Fcal.$
Let $\Seq{M_\xi : \xi \in \omega_2}$ be as in Definition \ref{fast} which witness that $C$ is fast for $T,X,Y, \Seq{b_\xi: \xi \in \omega_2}$.

By Lemma \ref{height} there is $q \leq p$ such that $\pphi_q = \pphi_p$ and
$\alpha_q = \alpha$ is above $ \sup \{\Delta (b_\xi, b_\zeta) \in \omega_1: \{b_\xi, b_\zeta\} \subset \{ b_\eta\}\cup \dom(\pphi_p) \cup \range (\pphi_p) \}$.
In particular for every $t \in U(\alpha) \cup V(\alpha)$
there is at most one $b \in \{b_\eta\}\cup \dom(\pphi_q) \cup \range (\pphi_q) $ with $t \in b$.
We find an extension $q' \leq q$ with $f_{q'}=f_q$ and $\dom(\pphi_p) \cup \{b_\eta\} = \dom(\pphi_{q'})$.
Let $t = b_\eta (\alpha)$.
Let $\mu$ be the smallest ordinal with $\eta \in M_{\mu}$.
Then $\mu$ is a successor ordinal or $\mu = 0$.

Let $b_\nu \in \Bcal(V_{f_q(t)}) \cap (M_\mu \setminus \bigcup_{\rho \in \mu}M_\rho)$ such that $\eta \in X$ iff $\nu \in Y$.
In order to see there is always such a $\nu$, first assume that $\eta \in X$. Then $ \mu \in \kappa$.
Condition (b) of Definition \ref{F} and elementarity imply that 
there is $b_\nu \in (Y \cap M_\mu) \setminus \bigcup_{\rho \in \mu} M_\rho$ such that $f_q(t) \in b_\nu$, as desired.
Now assume $\eta \notin X$.
Let $Y'= (\Bcal(V)\setminus Y) \cap (M_\mu \setminus \bigcup_{\rho \in \mu}M_\rho)$.
By elementarity of $M_\mu$ and Condition (b)
 we can find $b_\nu \in Y'$ such that 
$b_\nu \in \Bcal(V_{f_q(t)})$, as desired.

Define $q'$ by $\pphi_{q'}= \pphi_q \cup \{(\eta, \nu) \}$ and $f_{q'} = f_q$.
The way we chose $\nu$, makes \ref{respect}, \ref{restriction}, \ref{finite}, and \ref{con} obvious for $q'$.
Note that there is no $b \in \range(\pphi_q)$ with $f_q(t) \in b$.
(Assume for a contradiction that there is such a $b$.
Then by \ref{con} of Definition \ref{F}, $t \in [\pphi_q^{-1}(b)] \cap b_\eta$ which contradicts the choice of $q, \alpha_q$.)
In particular, $\pphi_{q'}$ is one-to-one. 
Moreover, for every $t \in U(\alpha) \cup V(\alpha)$
there is at most one $b \in \dom(\pphi_{q'}) \cup \range (\pphi_{q'}) $ with $t \in b$.
Since $f_q=f_{q'}$ preserves $<_{\lex}$ and \ref{con} holds for $q'$, property \ref{1-1} holds for $q'$.
So $q' \in \Fcal$ and it is an extension of $p$.
\end{proof}
Provided that $\omega_1$ is preserved and no new cofinal branch is added, Lemma \ref{branchDensity} asserts the following.
Let $G$ be the generic filter for $\Fcal_{X,Y}$ and $\pphi=\bigcup \{\pphi_p: p \in G \}$. Then both maps 
$\pphi \rest X : X \longrightarrow Y$ and 
 $\pphi : \Bcal(\bigcup X) \longrightarrow \Bcal(\bigcup Y)$ are $<_{\lex}$-isomorphisms. 
 
 \begin{lem}\label{F_omega}
Assume $T,X,Y,U,V,\Fcal$ are as in Definition \ref{F}. 
Then $\Fcal$ is $\Omega(T)$-complete. 
\end{lem}
\begin{proof}
Assume $\lambda > (2^{2^{\omega_1}})^+$ is a regular cardinal,
$M \prec H_{\lambda}$ is countable, 
$X,Y,T$ are in $M$, 
and $M$ captures all elements of $T$.
Let $\delta = M \cap \omega_1$ and 
$\Seq{p_n: n \in \omega}$ be a strictly decreasing 
$(M, \Fcal)$-generic sequence.
Let us use $f_n, \pphi_n, A_n, \alpha_n$ instead of $f_{p_n}, \pphi_{p_n}, A_{p_n}, \alpha_{p_n}$.
Define $A_p=\{ \delta \} \cup \bigcup_{n \in \omega}A_n$, 
$\pphi_p= \bigcup_{n \in \omega} \pphi_n $,
$f = \{(b(\delta), [\pphi_p(b)](\delta)) : b \in \Bcal(U) \cap M \}$ and  
$f_p= f \cup \bigcup_{n \in \omega} f_n$.

We show $p \in \Fcal$.
Note that by Lemma \ref{height} and genericity of $\Seq{p_n: n \in \omega}$, $\delta = \sup\{ \alpha_n : n \in \omega\}$.
In order to see $\dom(f) = U(\delta)$, let $t \in U(\delta)$.
Note that 
$T(\delta)=\{b_\xi(\delta): \xi \in \omega_2 \cap M \}$,
since $M$ captures all elements of $T$.
By elementarity  there is a unique $\xi \in \omega_2 \cap M$
such that $t \in b_\xi$.
Again by elementarity, $b_\xi \subset U$.
Lemma \ref{branchDensity} and the genericity of $\Seq{p_n: n \in \omega}$
implies that $\pphi_n(b_\xi)$ is defined for some $n \in \omega$ and hence 
$f(t)$ is defined.
Similar argument shows that $f: U(\delta) \longrightarrow V(\delta)$
is a bijection.
Since each $\pphi_n$ preserves $<_{\lex}$, 
$\pphi_p$ and $f$ preserve it too. 
In order to see $f_p$ preserves the tree order let $t \in b \in \Bcal(U) \cap M$ be as  and
$s <_T t$ be in $\dom(f_p)$. Let $n \in \omega$ such that  $\Ht(s) \in A_n$ and $b \in \dom(\pphi_n)$.
By \ref{con} of Definition \ref{F},   $$f_p(s)=f_n(s) \leq_T f_n(b(\alpha_n))= [\pphi_n(b)](\alpha_n)\leq_T [\pphi_n(b)](\delta)=f_p(t).$$
The rest of the conditions of Definition \ref{F} hold trivially.
\end{proof}

\begin{defn}\label{R}
Assume $(T,<_{\lex})$ is a lexicographically ordered  Kurepa tree 
and $L \subset \Bcal(T)$ is such that $\bigcup L$ is everywhere Kurepa. 
$R_L(T)$ is the poset consisting of all $p=(Z_p, A_p)$
with the following properties.
\begin{enumerate}
\item $Z_p \subset L$ is countable nonempty and disjoint from the set of all local 
end points of $\bigcup L$.
\item $A_p$ is a countable antichain of $T$ which does not intersect $\bigcup Z_p$.
\end{enumerate}
For $p$ and $q$ in $R_L$ let $q \leq p$ if $Z_p \subset Z_q$ and $A_p \subset A_q$. We use $R_L$ instead of $R_L(T)$ if there is no ambiguity.
\end{defn}
It is obvious that the forcing notion introduced in Definition \ref{R} is countably closed.
\begin{lem}\label{Z_closed}
Let $T,L,R_L$ be as in Definition \ref{R}.
For generic $G \subset R_L$, let $Z = \bigcup_{p \in G}Z_p$ and $A= \bigcup_{p \in G}A_p$.
Then for every $b \in \Bcal(T)$, either 
$b \in  Z$ or $b \cap A \neq \emptyset$.
In particular, $Z = \Bcal(\bigcup Z)$. 
\end{lem}
\begin{proof}
First recall that $\sigma$-closed posets do not add new cofinal 
branches to the $\omega_1$-trees of the ground model.
Fix $b \in \Bcal(T)$ and $p=(Z_p, A_p)$ in $R_L$. 
If $b \in Z_p$  then $p$ forces the 
conclusion of the lemma and we are done.
Assume $b \notin Z_p$ to see there is an extension of $p$ which 
forces the conclusion of the lemma.
Let $\alpha > \sup \{ \Delta(b , c) : c \in Z_p \} + 
\sup \{ \Ht(a) : a \in A \} $ 
be  a countable ordinal and $t = b (\alpha)$. 
Then $q=(Z_p, A_p \cup \{t \}) \leq p$
and it forces that $b \cap \dot{A} \neq \emptyset$, as desired.
\end{proof}

\begin{lem}\label{no_local_end_point}
Assume $L,Z,A, R_L$ are as in Lemma \ref{Z_closed}.
Then $\bigcup Z$ has no local end points.
\end{lem}
\begin{proof}
Assume for a contradiction that 
$p \in R_L$ and it forces that $\bigcup\dot{Z}$ has a local right end point. 
Lemma \ref{Z_closed} and the fact that $\sigma$-closed posets 
do not introduce new branches to the $\omega_1$-trees 
of the ground model we can replace $p$ by an extension if necessary 
such that for some 
$b \in Z_p$ and $t \in b$ the condition  $p$ forces that 
$b(\beta) = \max((\bigcup \dot{Z}) \cap T_t(\beta), <_{\lex})$
when $\beta \in \omega_1 \setminus \Ht(t)$.

Let $\alpha > \sup \{\Ht(a): a \in A_p \}$ be a countable ordinal.
Since 
$b$ is  not a local end point of $\bigcup L$, 
there are $b_0<_{\lex} b <_{\lex} b_1$ in $L$ such that 
$\alpha + \Ht(t) < \min \{ \Delta(b_0,b), \Delta(b_1, b) \}.$
Note that there are at most $\aleph_1$-many elements in $L$
which can be local end points of $\bigcup L$ and $\bigcup L$ is everywhere Kurepa.
So we can choose $b_0, b_1$ so that they are not local end points of $\bigcup L$.
Let $q=(Z_p \cup \{b_0, b_1 \}, A_p)$.
The way $\alpha, b_0,b_1$ are chosen guarantees that 
$A_p \cap b_i =\emptyset$ for $i \in 2$.
Hence $q$ is an extension of $p$ in $R_L$.
But $q$ forces that 
$b(\beta) \neq \max ((\bigcup \dot{Z}) \cap T_t(\beta), <_{\lex})$
when $\beta > \Delta(b_0, b_1)$, which is a contradiction.
The same argument shows that $\bigcup Z$ has no local left end points.
\end{proof}
\begin{lem}\label{Z_large}
Assume $L,Z,A, R_L$ are as in Lemma \ref{Z_closed}.
If $t \in b \in Z$ then the set $\{\beta \in \omega_2: b_\beta \in Z \cap \Bcal(T_t) \}$ is cofinal in $\omega_2$.
\end{lem}
\begin{proof}
Assume $\beta \in \omega_2$ and $p \in R_L$ such that $b \in Z_p$.
We show there is $\gamma \in \omega_2 \setminus \beta$ and $q \leq p$
such that $t \in b_\gamma \in Z_q$.
Since $b \in Z_p$ it is not a local end point of $\bigcup L$.
In $L$ find $b' <_{\lex} b$ with $t \in b'$. 
Since $\bigcup L$ is everywhere Kurepa
and there are at most $\aleph_1$-many elements of $\bigcup L$ that are local end points, 
there is $b_\gamma \in (b',b)$
which is not a local end point of $\bigcup L$ and $\gamma > \beta$.
Obviously $t \in b_\gamma$ and we are done.
\end{proof}
\begin{lem}
For every club $E$, $R_L$ satisfies $E$-cic for $\aleph_2$.
\end{lem}
\begin{proof}
Fix $X,h,M,N$ as in Definition \ref{S-cic} and $\delta = M\cap \omega_1 = N \cap \omega_1$.
Let $E \subset [X]^\omega$ be a club and 
$\langle p_n=(A_n, Z_n): n\in \omega \rangle $ be an $(M,\mathcal{P})$-generic sequence.
Then $h(A_n)= A_n \subset T\rest \delta$ and for all $b \in Z_n$, $\Delta(b , h(b)) > \delta $.
Therefore, $(\bigcup_{n \in \omega} A_n, \bigcup_{n \in \omega}(Z_n \cup h(Z_n))$ is a common lower bound as desired by Definition \ref{S-cic}.
\end{proof}

\begin{lem}\label{cc}
Assume $\CH$ and let $P=\Seq{P_i,Q_j:j<\mu, i \leq \mu}$
be a  countable support iteration of forcing notions where $Q_0 = Q$ of Definition \ref{Q} and $T$ be the generic tree 
for $Q_0$. Moreover assume that for all $j >0$
either $Q_j$ is of the form $\Fcal_{X,Y}(T)$ as in Definition \ref{F}
or $Q_j$ satisfies $\Omega(T)$-cic for $\aleph_2$.
Then $P$ has the $\aleph_2$-cc.
\end{lem}
\begin{proof}
Assume $\Seq{p_\xi : \xi \in \omega_2}$ is a sequence of conditions in 
$P$.
Let $\lambda > |2^P|^+$ be a regular cardinal and 
$\Seq{M_\xi: \xi \in \omega_2}$ be a sequence of countable 
elementary submodels of $H_\lambda$ such that 
$\{P,\Seq{p_\xi: \xi \in \omega_2}, \xi \}$ is an element of $ M_\xi$.
Define  $r : \omega_2 \longrightarrow \omega_2 $ by $r(\xi) = \sup ( M_\xi \cap \xi)$.
Obviously $r$ is regressive on the set of all ordinals of uncountable 
cofinality.
Let $W_0$ be a stationary subset of $\omega_2$ such that 
$r \rest W_0$ is constant and $\xi_0 > r(\xi_0)$ where $\xi_0 = \min(W_0)$.
Fix $W_1 \subset W_0$ with  $|W_1|= \aleph_2$ such that for all 
pairs $\xi < \eta$ in $W_1$, $\sup(M_\xi \cap \omega_2) < \eta$. 
In particular, for all pairs $\xi < \eta $ in $W_1$, 
\begin{equation}\label{ordered_models}
\min((M_\eta \setminus M_\xi) \cap \omega_2) > \sup(M_\xi \cap \omega_2).
\end{equation}

$\CH$ implies that $\{(M_\xi, p_\xi): \xi \in W_1 \}$ has at most $\aleph_1$-many isomorphism 
types.
Use $\CH$ and find  $W_2 \subset  W_1$ with $|W_2|=\aleph_2$
such that 
$\{ M_\xi : \xi \in W_2 \}$ forms a $\Delta$-system with root $\Delta$
and 
for all pairs $\xi <\eta$ in $W_2$, $(M_\xi,\Delta, p_\xi) \cong_{h_{\xi \eta}}
(M_\eta, \Delta, p_\eta)$.
Note that by the axiom of extensionality, the isomorphism $h_{\xi \eta}$ from 
$M_\xi$ to $M_\eta$ is uniquely determined by $M_\xi$ and  $M_\eta$.
In particular for all pairs $\xi < \eta$ in $W_2$
\begin{equation}
(M_\xi, p_\xi) \cong_{h_{\xi \eta}}
(M_\eta, p_\eta) \textrm{ and }
h_{\xi \eta} \rest (M_\xi \cap M_\eta) = \id_{(M_\xi \cap M_\eta)}.
\end{equation}

We show that the elements in $\{ p_\xi: \xi \in W_2 \}$ are pairwise 
compatible. In order to see this fix $\xi < \eta$ in $W_2$ and let 
$h = h_{\xi \eta}$.
Let $\Seq{p_\xi^n: n \in \omega}$ be a decreasing $(M_\xi, P)$-generic
sequence with $p_\xi^0 = p_\xi$.
Observe that $\Seq{p_\eta^n = h(p_\xi^n) : n \in \omega}$ 
is a decreasing $(M_\eta,P)$-generic sequence.
Our aim is to show that there is $p \in P$ which is a 
common lower bound for $\Seq{p_\xi^n: n \in \omega}$ and 
$\Seq{p_\eta^n: n \in \omega}$. 

We proceed by defining $p(\beta)$ for every $\beta \in \mu$ inductively.
Let $\delta = M_\xi \cap \omega_1$.
If $\beta \notin M_\xi \cup M_\eta$ let $p(\beta)$ be the trivial condition.
First assume $\beta = 0$.
Recall that $p_\xi^n(0)$ and $p_\eta^n (0)$ are conditions in $Q.$
Note that $\bigcup_{n \in \omega}T_{p_\xi^n (0)} =
\bigcup_{n \in \omega}T_{p_\eta^n (0)}$
since  $h$ is an isomorphism  from $ (M_\xi, \Seq{p_\xi^n : n \in \omega})$ to 
$ (M_\eta, \Seq{p_\eta^n : n \in \omega})$.
Let $R^- = \bigcup_{n \in \omega}T_{p_\xi^n (0)}$.
For every $\alpha \in M_\xi \cap \omega_2$ and $n \in \omega$ with 
$\alpha \in \dom(b_{p_\xi^n(0)})$,
$b_{p_\xi^n(0)}(\alpha)=h(b_{p_\xi^n(0)}(\alpha))= b_{p_\eta^n(0)}(h(\alpha))$, 
since $h$ fixes the intersection.
By the genericity assumptions, 
$\Seq{b_{p_\xi^n(0)}(\alpha): n \in \omega \textrm{ and } \alpha \in \dom(b_{p_\xi^n(0)})}$, 
which equals $\Seq{b_{p_\eta^n(0)}(h(\alpha)) : n \in \omega
\textrm{ and } \alpha \in \dom(b_{p_\xi^n(0)})}$, 
is a cofinal branch of $R^-$.
Let $R$ be a countable tree of height $\delta+1$ such that $R\rest \delta = R^-$ and every cofinal branch $c$ of  $R^-$ has a unique top element if
$\Seq{b_{p_\xi^n(0)}(\alpha): n \in \omega \textrm{ and } \alpha \in \dom(b_{p_\xi^n(0)})}$
is cofinal in $c$
for some $\alpha \in M_\xi \cap \omega_2$.
If there is no $\alpha \in M_\xi \cap \omega_2$ such that the sequence 
$\Seq{b_{p_\xi^n(0)}(\alpha): n \in \omega \textrm{ and } \alpha \in \dom(b_{p_\xi^n(0)})}$ is cofinal in $c$ 
we do not extend $c$ in $R$.
Let $T_{p(0)}$ be a lexicographically ordered countable tree of height 
$\delta +2$ such that $T_{p(0)} \rest (\delta+1) = R$ and 
for every $t \in R(\delta)$ the set $\{ s \in T_{p(0)}: s >t \}$
is isomorphic to the rationals when it is considered with the lexicographical order from $T_{p(0)}$.

Let $\dom(b_{p(0)}) = (M_\xi \cup M_\eta) \cap \omega_2$.
Let $A=(M_\xi \cap M_\eta) \cap \omega_2$, $B=(M_\xi \setminus M_\eta) \cap \omega_2$, and $C = (M_\eta \setminus M_\xi) \cap \omega_2$, which form a partition of $(M_\xi \cup M_\eta) \cap \omega_2$.
If $\alpha \in A \cup C$, let $b_{p(0)}(\alpha)$ be any element in $T_{p(0)}(\delta+1)$ which is above all elements in $\{b_{p_\eta^n(0)}(\alpha) : n \in \omega \textrm{ and } \alpha \in \dom(b_{p_\eta^n(0)}) \}$.
For $\alpha \in B$ let  
$t \in R(\delta)$ such that $t <_{T_{p(0)}} b_{p(0)}(h(\alpha))$. 
Now let $b_{p(0)} (\alpha)$ be ant element in $T_{p(0)}(\delta+1)$
which is above $t$ with $b_{p(0)}(\alpha) <_{\lex} b_{p(0)}(h(\alpha))$.
It is easy to observe that 
$p(0) =(T_{p(0)}, b_{p(0)}) \in Q$ is a common lower bound for 
both sequences $\Seq{p_\xi^n(0): n \in \omega}$ and $\Seq{p_\eta^n(0): n \in \omega}$.

For $\beta \in \mu \setminus 1$ in $M_\xi \cup M_\eta$ assume $p\rest \beta$ is given 
such that $p(0)$ is as above and $p \rest \beta$ is a common 
lower bound for $\Seq{p_\xi^n \rest \beta: n \in \omega}$ and $\Seq{p_\eta^n \rest \beta: n \in \omega}$.
In what follows, $\dot{G}$ is the canonical $P_\beta$-name for its generic filter. 
It suffices to find $p(\beta)$ such that $p \rest (\beta +1)$ is a common lower bound for $\Seq{p_\xi^n \rest (\beta +1): n \in \omega}$ and $\Seq{p_\eta^n \rest (\beta +1 ): n \in \omega}$ in the following cases:
\begin{itemize}
\item[($a_0$)] $\beta \in M_\xi \cap M_\eta$ and $\dot{Q}_\beta$ is a $P_\beta$-name for a poset
which satisfies $\Omega(T)$-cic for $\aleph_2$, 
\item[($a_1$)] $\beta \in M_\xi \cap M_\eta$ and $\dot{Q}_\beta$
is a $P_\beta$-name for some $\Fcal_{X,Y}$,
\item[($b_0$)] $\beta \in M_\xi \setminus M_\eta$ and $\dot{Q}_\beta$ is a $P_\beta$-name for a poset
which satisfies $\Omega(T)$-cic for $\aleph_2$, 
\item[($b_1$)] $\beta \in M_\xi \setminus M_\eta$ and $\dot{Q}_\beta$
is a $P_\beta$-name for some $\Fcal_{X,Y}$,
\item[($c_0$)] $\beta \in M_\eta \setminus M_\xi$ and $\dot{Q}_\beta$ is a $P_\beta$-name for a poset
which satisfies $\Omega(T)$-cic for $\aleph_2$, 
\item[($c_1$)] $\beta \in M_\eta \setminus M_\xi$ and $\dot{Q}_\beta$
is a $P_\beta$-name for some $\Fcal_{X,Y}$.
\end{itemize}

Assume ($a_0$).
Observe that $p \rest \beta$ is both $(M_\xi, P_\beta)$-generic and $(M_\eta,P_\beta)$-generic and it decides both $\dot{Q}_\beta \cap M_\xi[\dot{G}]$, $\dot{Q}_\beta \cap M_\xi[\dot{G}]$.
Also $p \rest \beta$ forces that $\delta = M_\xi[\dot{G}] \cap \omega_1 = M_\eta[\dot{G}] \cap \omega_1$.
The way $R(\delta)$ was defined guarantees that 
$ p\rest \beta \Vdash \{M_\xi[\dot{G}] \cap \Bcal(\dot{T}), 
M_\eta[\dot{G}] \cap \Bcal(\dot{T}) \} \subset \Omega(\dot{T})$.
Moreover for all $q \in M_\xi \cap P_\beta$,
 $(p \rest \beta) \Vdash_{P_\beta} ``q \in \dot{G} \leftrightarrow h(q) \in \dot{G}."$
In particular, 
$p \rest \beta$ forces $h[\dot{G}]: M_\xi[\dot{G}] \longrightarrow M_\eta[\dot{G}]$ -- which is defined by $\tau_G \mapsto h(\tau)_G$ -- is an isomorphism. 
Note that $h[\dot{G}]$ is decided by $p \rest \beta$.
It is forced by $p \rest \beta$ that $h[\dot{G}]$ fixes $M_\xi[\dot{G}] \cap M_\eta[\dot{G}]$ since it forces that
$(M_\xi \cap M_\eta) \cap \lambda = M_\xi[\dot{G}] \cap M_\eta[\dot{G}] \cap \lambda$ -- which is equivalent to the assertion that
$p\rest \beta$ is $(M_\xi \cap M_\eta, P_\beta)$-generic. 
Recall that $\dot{Q}_\beta$ is a $P_\beta$-name for  a poset which 
satisfies $\Omega(T)$-cic for $\aleph_2$.
Therefore, $p\rest \beta$ forces that the sequences 
$\Seq{p_\xi^n(\beta): n \in \omega}$,
$\Seq{p_\eta^n(\beta): n \in \omega}$
have a common lower bound which we take to be $p(\beta)$.

Assume ($a_1$). 
Let $\dot{X}, \dot{Y}, \dot{U}, \dot{V}, \dot{C}$ in $M_\xi \cap M_\eta$
be the $P_\beta$-names which correspond to the objects in 
Definition \ref{F}.
Since $p \rest \beta$ is $(M, P_\beta)$-generic for $M \in \{ M_\xi, M_\eta \}$, it decides 
$\Seq{p_\xi^n(\beta): n \in \omega}$,
$\Seq{p_\eta^n(\beta): n \in \omega}$,
$|\dot{X}|, |\dot{Y}|, \iota [\dot{X}] \cap M, \iota[\dot{Y}]\cap M,$ $ \iota[\Bcal(\dot{U})] \cap M$, $\iota[\Bcal(\dot{V})] \cap M$, $\dot{U} \cap R^-, 
\dot{V} \cap R^-$ for $M \in \{M_\xi, M_\eta \}$. 
Since $p\rest \beta$ forces that $M_\xi, M_\eta$ capture all elements in 
$\dot{T}$, it decides $U(\delta)=\dot{U} \cap R(\delta)$ and $V(\delta)=\dot{V} \cap R(\delta)$.
Let $\kappa \in M_\xi \cap M_\eta$ be the cardinal which is forced by $p \rest \beta$
to be the cardinality of $\dot{X}, \dot{Y}$. 

Define $A_{p(\beta)} = \{ \delta \} \cup \bigcup_{n \in \omega} A_{p_\xi^n(\beta)}$,
$\psi_{p(\beta)} = [\bigcup_{n \in \omega} \psi_{p_\xi^n(\beta)}] \cup [\bigcup_{n \in \omega} \psi_{p_\eta^n(\beta)}]$,
$f=\{([b_{p(0)}(\alpha)](\delta), [b_{p(0)}(\psi_{p(\beta)}(\alpha))](\delta)): \alpha \in \omega_2 \cap M_\xi  \}$, and 
$f_{p(\beta)} = f \cup \bigcup_{n \in \omega} f_{p_\xi^n(\beta)}$.
Observe that $f$ is a bijection between $U(\delta)$ and $V(\delta)$
which preserves the lexicographic order which is forced by $p\rest \beta.$ 
First we show that $\psi_{p(\beta)}$ is a function.
Assume $\alpha \in \dom(\psi_{p(\beta)})$.
If $\alpha \in M_\xi \setminus M_\eta$ or $\alpha \in M_\eta \setminus M_\xi$ there is nothing to prove.
So assume $\alpha \in M_\xi \cap M_\eta.$
Observe that $p\rest \beta$ forces $\alpha \in \sup(M_\xi \cap M_\eta \cap \omega_2) \in \dot{C}$.
Then $\psi_{p_\xi^n(\beta)}(\alpha) \in M_\xi \cap M_\eta$ if $\alpha \in \dom(\psi_{p_\xi^n})$.
Since $h$ fixes $M_\xi \cap M_\eta$, $\psi_{p_\xi^n(\beta)}(\alpha) = \psi_{p_\eta^n(\beta)}(\alpha) $  if $\alpha \in \dom(\psi_{p_\xi^n})$.
Similar argument shows that $\psi_{p(\beta)}$ is one-to-one.

Now we show $p \rest \beta$ forces that $\pphi_{p(\beta)} = \iota^{-1} \circ \psi_{p(\beta)} \circ \iota$ preserves the lexicographic order of $\dot{T}$. 
Note that $p \rest \beta$ decides $\dot{b}_\alpha$ up to height 
$\delta + 1$ and the truth of $``\dot{b}_\alpha <_{\lex} \dot{b}_\gamma"$ for all $\alpha, \gamma $ in $ \omega_2 \cap (M_\xi \cup M_\eta)$.
Assume $p \rest \beta$ forces that $\dot{b}_\alpha <_{\lex} 
\dot{b}_\gamma$. 
We show $p \rest \beta \Vdash ``\dot{b}_{\psi_{p(\beta)}(\alpha)}
<_{\lex} \dot{b}_{\psi_{p(\beta)}(\gamma)}."$
If $\alpha, \gamma$ are both in one of $M_\xi, M_\eta$ there is nothing to show.
If $p \rest \beta \Vdash `` \dot{b}_\alpha , \dot{b}_\gamma$ agree up to height $\delta"$ then 
$\gamma = h(\alpha)$
and the desired assertion is forced by the definition of $p(0)$.
If $p \rest \beta \Vdash ``\Delta(\dot{b}_\alpha , \dot{b}_\gamma) < \delta"$ then the desired assertion follows from \ref{1-1} and \ref{con} of Definition \ref{F} and the fact that $f_{p_\eta^n(\beta)} = h(f_{p_\xi^n(\beta)}) = f_{p_\xi^n(\beta)}  $  for all $n \in \omega$.
Properties \ref{respect}, \ref{restriction}, \ref{finite}, \ref{con} 
of Definition \ref{F}
are immediate from the construction. 
This makes $p(\beta)$ a lower bound for the sequences 
$\Seq{p_\xi^n(\beta): n \in \omega}$ and 
$\Seq{p_\eta^n(\beta): n \in \omega}$.

Assume $(b_0)$. 
Note that $\Omega(T)$-cic implies $\Omega(T)$-completeness.
The argument in this case is similar to the one in $(a_0)$ except that 
we only need to find a lower bound $p(\beta)$ for the sequence $\Seq{p_\xi^n(\beta): n \in \omega}$ and we do not need to worry about two models and the 
isomorphism between them. 
Note that $p(\beta)$ which is found in this way is a common lower bound for both sequences $\Seq{p_\xi^n(\beta): n \in \omega}$ and 
$\Seq{p_\eta^n(\beta): n \in \omega}$ since the conditions in 
$\Seq{p_\eta^n(\beta): n \in \omega}$
are all trivial conditions.

The argument for $(c_0)$ is similar to $(a_0)$ in the same way.
Also the arguments for $(a_1), (b_1)$ and $(c_1)$ 
are similar by the same reasoning.
Moreover, the ordinals $\beta$ which fall into cases $(a_0), (a_1)$
are all smaller than the ordinals $\beta$ for which the assumptions 
$(b_0),(b_1), (c_0)$ or $(c_1)$ holds. 
This shows that we can find a common lower bound 
$p(\beta)$ for the sequences $\Seq{p_\xi^n(\beta): n \in \omega}$ and 
$\Seq{p_\eta^n(\beta): n \in \omega}$
inductively, as desired.
\end{proof}

Note that the forcing notion $\Fcal_{X,Y}$ which appear in the iteration 
in Lemma \ref{cc} do not satisfy $\Omega(T)$-cic because of the lexicographic order of the cofinal branches of $T$. 
However,  the tree $T$ is built generically and consequently the lexicographic order of the 
cofinal branches -- as specified in the construction of $p(0)$ in the proof of 
Lemma \ref{cc} -- can be robust often enough so that the behavior of 
$\Fcal_{X,Y}$ is close the forcing notions which satisfy $\Omega(T)$-cic.
This leads us to Theorems \ref{main} and \ref{min} which we are going to prove.
\begin{proof}[Proof of Theorem \ref{main}]
Let $\textsc{V}$ be a model of $\GCH$.
Let $P$ be the countable support iteration of forcing notions 
of length $\omega_3$ as in Lemma \ref{cc}
such that whenever $ \Bcal(T)= X \cup Y$ is a partition and 
$|\Bcal(T_t) \cap X|= |\Bcal(T_t) \cap Y| = \aleph_2$ for 
all $t \in T$ then $\Fcal_{X,Y}$ appears in the iteration.
Note that $T$ has no local end points and $X,Y$ satisfy 
the hypotheses of Definition \ref{F}.
Lemmas \ref{F_omega}, \ref{cc}, \ref{No New branch} and the usual bookkeeping arguments 
allow us to make sure that for all possible 
$X,Y$ described above $\Fcal_{X,Y}$ appears in the iteration.
Lemmas \ref{No A subtree} and \ref{iteration} imply that $T$ has no Aronszajn subtree 
in the extension of $\textsc{V}$ by $P$.
Lemmas \ref{F_omega} and \ref{cc} imply that $P$ preserves $\omega_1, \omega_2$. 
By Lemma \ref{branchDensity}, 
$\Fcal_{X,Y}$ makes $X,Y$ order isomorphic.
By Lemma \ref{height}, 
$\Fcal_{X,Y}$ adds a club embedding which makes $X,Y$ homeomorphic. 
\end{proof}

\begin{proof}[Proof of Theorem \ref{min}]
Let $\textsc{V}$ be a model of $\GCH$.
Let $P$ be the countable support iteration of forcing notions 
of length $\omega_3$ as in Lemma \ref{cc}
such that whenever 
$D$ is a dense subset $(\Bcal(T), <_{\lex})$ of size $\aleph_1$ and
$L \subset \Bcal(T)$ and $\bigcup L$ is everywhere 
Kurepa then $R_L * \Fcal_{D,Y_L}$ appears in the iteration, where
$Y_L \subset Z_L$ is dense, $ |Y_L| = \aleph_1$ and $Z_L$
is the generic object that is added by $R_L$.
By Lemmas \ref{Z_closed}, \ref{no_local_end_point} and \ref{Z_large}, 
$D$ and  $Y_L$ satisfy the hypotheses of Definition  \ref{F}. 
$\aleph_1$ and $\aleph_2$ are preserved in the same way as above.
By Lemma \ref{Z_closed}, $\Fcal_{D,Y_L}$ adds an order isomorphism 
from $\Bcal(T)$ to $Z_L$ via the dense subsets.
Since every Kurepa line contains an everywhere Kurepa suborder, 
in the extension of $\textsc{V}$ by $P$ there is an embedding 
from $\Bcal(T)$ to $L$, when $L$ is a Kurepa suborder of $(\Bcal(T), <_{\lex})$.
\end{proof}

\def\Dbar{\leavevmode\lower.6ex\hbox to 0pt{\hskip-.23ex \accent"16\hss}D}

\end{document}